\documentclass[11pt, a4paper]{amsart}

\usepackage{amsfonts}
\usepackage{amssymb}
\usepackage{amsmath, amsthm,color}
\usepackage{hyperref}
\usepackage{pdfsync}
\usepackage{bbm, dsfont}
\usepackage[margin=0.96in]{geometry}

\newcommand{\R}{\mathbb{R}}

\newtheorem{theorem}{Theorem}[section]

\newtheorem{lemma}[theorem]{Lemma}

\numberwithin{equation}{section}

\begin{document}
\title{Bobkov's inequality via optimal control theory}
\author[F.~Barthe, P.~Ivanisvili]{F.~Barthe,  P.~Ivanisvili}
\thanks{This paper is  based upon work supported by the National Science Foundation under Grant No. DMS-1440140 while two of the authors were in residence at the Mathematical Sciences Research Institute in Berkeley, California, during the Fall 2017 semester. }
\address{Institut de Math\'ematiques de Toulouse; UMR 5219 \\ Universit\'e de Toulouse; CNRS
\\
 France}
\email{franck.barthe@math.univ-toulouse.fr \textrm{(F.\ Barthe)}}

\address{Department of Mathematics, Princeton University; MSRI; UC Irvine, CA, USA}
\email{paata.ivanisvili@princeton.edu \textrm{(P.\ Ivanisvili)}}

\makeatletter
\@namedef{subjclassname@2010}{
  \textup{2010} Mathematics Subject Classification}
\makeatother
\begin{abstract} 
We give the simple proof of Bobkov's inequality using the arguments of dynamical programming principle. As a byproduct of the method we obtain a characterization of optimizers. 
\end{abstract}
\maketitle 
\section{Bobkov's inequality}
Bobkov's inequality~\cite{Bobp} states that 
\begin{align}\label{bob}
\int_{\mathbb{R}^{n}} \sqrt{I^{2}(f)+|\nabla f|^{2}}\, d\gamma^n \geq I\left(\int_{\mathbb{R}^{n}} f\, d\gamma^n \right)
\end{align} 
holds for any smooth $f :\mathbb{R}^{n} \to [0,1]$, where $d\gamma^n(x) = \frac{e^{-|x|^{2}/2}}{(\sqrt{2\pi})^{n}}dx$ is the standard Gaussian measure on $\mathbb{R}^{n}$,   $I(x) = \varphi(\Phi^{-1}(x))$, $\Phi(t) = \gamma^1((-\infty,t])$ and $\varphi(t)=\Phi'(t)$.
We simply write $\gamma$ for $\gamma^1$. 
This functional inequality implies the sharp isoperimetric  inequality for the gaussian measure $\gamma^n$ (\cite{bore75bmig,sudat78ephs,ehrh83seg}), and has led to far-reaching
extensions~\cite{BALE}. Bobkov's original proof of \eqref{bob} relies on a delicate two-point inequality and the central limit theorem. The inequality could be reproved by interpolation along the Ornstein-Uhlenbeck semigroup \cite{Ledoux,BALE} and by stochastic calculus \cite{BAMU}. Actually, \eqref{bob} can be deduced by applying the gaussian isoperimetric inequality (in $\mathbb R^{n+1}$) to the subgraph of the function $\Phi^{-1}(f)$ (but the main interest of \eqref{bob} is to give a more flexible proof of it). The calculation of the gaussian boundary measure of a subgraph can be found in Ehrhard's paper  \cite{E}.

 In this short paper we give a new proof of Bobkov's inequality using the standard dynamical programming principle.  A similar approach was used in \cite{AdCl,Adam} for Log-Sobolev and Hardy type inequalities. 
As a byproduct of the method, we easily obtain a characterization of  smooth optimizers in (\ref{bob}). The next section presents a direct proof, which is based on an explicit solution of a partial differential equation. Explanations about the origin of this PDE, in relation with dynamic programming, are given afterwards.

\section{The proof: Hamilton--Jacobi--Bellman PDE}
Given any $t, p\in \mathbb{R}$,  and $y$ with $0<y<\Phi(t)$, we claim that the following equation 
\begin{align}\label{franck}
\int_{-\infty}^{t} \Phi\big((s-t)a+p\big) \varphi(s) \, ds=y
\end{align}
has a unique $C^{1}$ solution $a=a(t,p,y)$.  Indeed, notice that by Fubini's theorem the left hand side of (\ref{franck}) represents the gaussian measure of the ``truncated halfspaces'', i.e.,  
\begin{align}\label{f2}
\gamma^2 \left( \{(s,u) \in \mathbb{R}^{2} \,  : \, s\leq t \mbox{ and } u \leq  (s-t)a+p\}\right)=y.
\end{align}
Clearly the left hand side of (\ref{f2}) is continuously decreasing in $a$,  when $a \to -\infty$ it tends to $\Phi(t)$, and when $a \to +\infty$ it goes to zero. Since $0<y<\Phi(t)$ we see that there exists a unique solution $a=a(t,p,y)$. The fact that $a \in C^{1}$ follows from the implicit function theorem (see the computations of partial derivatives below). 

\begin{lemma}\label{HJB}
Let 
\begin{align}\label{oprM}
M(t,p,y):=\varphi\left(\frac{p-a(t,p,y)\,t}{\sqrt{1+a^{2}(t,p,y)}}\right) \Phi\left(\frac{t+a(t,p,y)\, p}{\sqrt{1+a^{2}(t,p,y)}}\right)\quad \text{for} \quad p,t \in \mathbb{R}, \; 0<y<\Phi(t). 
\end{align}
We have 
\begin{align}\label{HJB-0}
\sqrt{\varphi^{2}(t) \varphi^{2}(p) - M_{p}^{2}} =M_{t}+\Phi(p) \varphi(t) M_{y}\, ,
\end{align}
where $M_{t}, M_{p}$ and $M_{y}$ denote the partial derivatives.
\end{lemma}

\begin{proof}
The derivative of the left-hand side of \eqref{franck} with respect to the variable $a$ is equal to 
\[ \int_{-\infty}^{t} \varphi((s-t)a+p)\varphi(s)(s-t)ds, \]
which is strictly negative. Therefore we can apply the implicit function theorem, and get a function
$a=a(t,p,y)$.
Next we compute the  partial derivatives of $a$. Differentiating  (\ref{franck}) with respect to  $t$  gives 
\begin{align}\label{car1}
\Phi(p)\varphi(t)-a\int_{-\infty}^{t} \varphi((s-t)a+p)\varphi(s)ds + a_{t}\int_{-\infty}^{t} \varphi((s-t)a+p)\varphi(s)(s-t)ds=0.
\end{align}
The latter two integrals can be computed directly:
\begin{align*}
&\int_{-\infty}^{t} \varphi((s-t)a+p)\varphi(s)ds = \frac{1}{\sqrt{1+a^{2}}}\varphi\left(\frac{p-at}{\sqrt{1+a^{2}}}\right) \Phi\left(\frac{t+ap}{\sqrt{1+a^{2}}}\right) = \frac{M}{\sqrt{1+a^{2}}}; \\
&\int_{-\infty}^{t} \varphi((s-t)a+p)\varphi(s)(s-t)ds= \\
=&- \frac{\varphi(\frac{p-at}{\sqrt{1+a^{2}}}) \varphi(\frac{t+ap}{\sqrt{1+a^{2}}})}{1+a^{2}} - \frac{a(p-at)}{(1+a^{2})^{3/2}}\varphi\left(\frac{p-at}{\sqrt{1+a^{2}}}\right)\Phi\left(\frac{t+ap}{\sqrt{1+a^{2}}}\right) - t \frac{M}{\sqrt{1+a^{2}}}=
\\=&-\frac{\varphi(\frac{p-at}{\sqrt{1+a^{2}}}) \varphi(\frac{t+ap}{\sqrt{1+a^{2}}})}{1+a^{2}}-\frac{(t+ap)}{(1+a^{2})^{3/2}}M. 
\end{align*}
These formulas suggest to introduce two auxiliary functions:
\begin{align}\label{ptPQ}
P:=\frac{p-at}{\sqrt{1+a^{2}}} \quad \text{and}  \quad Q:=\frac{t+ap}{\sqrt{1+a^{2}}}.
\end{align}
Then $M(t,p,y) = \varphi(P) \Phi(Q)$,
and the latter two integrals become
\begin{align}
&\int_{-\infty}^{t} \varphi((s-t)a+p)\varphi(s)ds =\frac{\varphi(P)\Phi(Q)}{\sqrt{1+a^2}}\label{pir1}\\
&\int_{-\infty}^{t} \varphi((s-t)a+p)\varphi(s)(s-t)ds= 
 -\frac{\varphi(P) (\varphi(Q)+Q \Phi(Q))}{1+a^{2}}.\label{meor2}
\end{align}
Thus using (\ref{car1}), (\ref{pir1}) and (\ref{meor2}) we obtain  
\begin{align*}
a_{t} = \frac{\Phi(p)\varphi(t)-a \int_{-\infty}^{t} \varphi((s-t)a+p)\varphi(s)ds}{  -\int_{-\infty}^{t} \varphi((s-t)a+p)\varphi(s)(s-t)ds}=(1+a^{2})\frac{\Phi(p)\varphi(t) - \frac{a}{\sqrt{1+a^{2}}}\varphi(P)\Phi(Q)}{\varphi(P)(\varphi(Q)+Q\Phi(Q))}.
\end{align*}
In a similar way we compute 
\begin{align*}
a_{y} =\frac{1}{\int_{-\infty}^{t} \varphi((s-t)a+p)\varphi(s)(s-t)ds} = \frac{-(1+a^{2})}{\varphi(P)(\varphi(Q)+Q\Phi(Q))},
\end{align*}
and 
\begin{align*}
a_{p}=\frac{-\int_{-\infty}^{t} \varphi((s-t)a+p)\varphi(s)ds}{\int_{-\infty}^{t} \varphi((s-t)a+p)\varphi(s)(s-t)ds} = \frac{\Phi(Q) \sqrt{1+a^{2}}}{\varphi(Q)+Q\Phi(Q)}.
\end{align*}

Now let us compute the partial derivatives of $M = \varphi(P) \Phi(Q)$. First we compute the partial derivatives  of $P$ and $Q$. We have 
\begin{align*}
&P_{t} = \frac{\partial }{\partial t} \left(\frac{p-at}{\sqrt{1+a^{2}}}\right) = -\frac{a}{\sqrt{1+a^{2}}}  -\frac{a_{t}}{1+a^{2}}Q; \quad Q_{t} = \frac{1}{\sqrt{1+a^{2}}} + \frac{a_{t}}{1+a^{2}}P;\\
&P_{p} = \frac{1}{\sqrt{1+a^{2}}} - \frac{a_{p}}{1+a^{2}}Q; \quad Q_{p} = \frac{a}{\sqrt{1+a^{2}}} + \frac{a_{p}}{1+a^{2}}P;\\
&P_{y} = \frac{-a_{y}}{1+a^{2}}Q; \quad Q_{y} = \frac{a_{y}}{1+a^{2}}P.
\end{align*}
Therefore we have 
\begin{align*}
&M_{t} = \frac{aP\varphi(P)\Phi(Q) + \varphi(P)\varphi(Q)}{\sqrt{1+a^{2}}}+\frac{P\varphi(P)a_{t}}{1+a^{2}}(Q\Phi(Q) + \varphi(Q))=\frac{\varphi(P) \varphi(Q)}{\sqrt{1+a^{2}}} + P \varphi(t) \Phi(p); \\
&M_{p} = \frac{-P\varphi(P)\Phi(Q)+\varphi(P)\varphi(Q)a}{\sqrt{1+a^{2}}} + \frac{a_{p}P\varphi(P)}{1+a^{2}}(Q\Phi(Q)+\varphi(Q)) = \frac{\varphi(P)\varphi(Q) a}{\sqrt{1+a^{2}}}; \\
&M_{y} = (\varphi(P) \Phi(P))_{y}=-P\varphi(P) \Phi(Q) P_{y} + \varphi(p) \varphi(Q) Q_{y} = 
\frac{a_{y}}{1+a^{2}}\varphi(P) P ( Q\Phi(Q)+\varphi(Q))=-P. 
\end{align*}
Thus
\begin{align}\label{bro}
M_{t}+\Phi(p)\varphi(t) M_{y} = \frac{\varphi(P)\varphi(Q)}{\sqrt{1+a^{2}}}, \quad \varphi^{2}(t) \varphi^{2}(p) - M_{p}^{2} = \varphi^{2}(P) \varphi^{2}(Q)  \frac{1}{1+a^{2}},
\end{align}
where in the last equality we have used that $\varphi(p) \varphi(t) = \varphi(P) \varphi(Q)$, a direct consequence of \eqref{ptPQ}. Identities in (\ref{bro}) imply (\ref{HJB-0}), and thereby the lemma is proved. 

Let us point out, for further use, that the latter identity satisfied by $\varphi$ gives that
\begin{align}\label{Mp}
M_p=\frac{a}{\sqrt{1+a^2}} \varphi(p)\varphi(t). 
\end{align}
\end{proof}
\begin{lemma}\label{endpoints}
Let $M$ be defined as in (\ref{oprM}), and let $f : \mathbb{R} \to (0,1)$ be any $C^{1}$ smooth function.
 Then 
\begin{align}
&\lim_{t \to -\infty} M\left(t,\Phi^{-1}(f(t)),\int_{-\infty}^{t} f d\gamma\right) = 0;\label{minus}\\
&\lim_{t \to \infty} M\left( t,\Phi^{-1}(f(t)),\int_{-\infty}^{t} f d\gamma\right) = I\left(\int_{\mathbb{R}} f d\gamma \right). \label{plus}
\end{align} 
\end{lemma}
\begin{proof}
Here we set (omitting variables) $p=p(t):= \Phi^{-1}(f(t))$, $y=y(t):=\int_{-\infty}^{t} f d\gamma$
and 
\[M=M(t,p,y)=\varphi\left(\frac{p-at}{\sqrt{1+a^2}}\right) \Phi\left(\frac{t+ap}{\sqrt{1+a^2}}
 \right),\] 
 where $a=a(t,p,y)$ is defined implicitly by \eqref{franck}.

First we check (\ref{minus}). Let $\varepsilon>0$ be an arbitrary positive number.
Then there exists $A$ such that: $|u|\ge A \Longrightarrow \varphi(u)\le \varepsilon$. If $\left| \frac{p-at}{\sqrt{1+a^2}}\right|\ge A$ then clearly $|M|\le \varepsilon$.
 On the contrary, if  $\theta=\theta(t):=\frac{p-at}{\sqrt{1+a^{2}}}$ verifies $|\theta(t)|< A$ then 
 \[  \frac{t+ap}{\sqrt{1+a^2}}=t\sqrt{1+a^2}+\theta a \le t\sqrt{1+a^2}+A |a|\le (t+A)\sqrt{1+a^2}, \]
which tends to $-\infty$ when $t\to -\infty$. Therefore, for $t$ sufficiently negative,
$$|M|\le \Phi\left(  \frac{t+ap}{\sqrt{1+a^2}}\right)\le \varepsilon.$$ Since
$\varepsilon>0$ was arbitrary, we have shown that 
\begin{align*}
\lim_{t \to -\infty} M\left(t,\Phi^{-1}(f(t)),\int_{-\infty}^{t} f d\gamma\right)  =  0. 
\end{align*}

To verify (\ref{plus}) we notice that  (\ref{franck}) implies 
\begin{align*}
y+\int_{t}^{\infty}\Phi((s-t)a+p)\varphi(s)ds = \int_{-\infty}^{\infty}\Phi((s-t)a+p)\varphi(s)ds = \Phi\left( \frac{p-at}{\sqrt{1+a^{2}}}\right).
\end{align*}
Therefore we obtain 
\begin{align*}
\lim_{t \to \infty}\frac{p-at}{\sqrt{1+a^{2}}}=\lim_{t \to \infty} \Phi^{-1}\left(\int_{-\infty}^{t} fd\gamma + \int_{t}^{\infty} \Phi((s-t)a+p)\varphi(s)ds \right) =\Phi^{-1}\left(\int_{-\infty}^{\infty} fd\gamma \right)
\end{align*}
regardless of the values of the function $a$. Since $f$ takes values in $(0,1)$, we have proved
that the function $\theta(t)=\frac{p-at}{\sqrt{1+a^{2}}}$ has a (finite) limit when $t$ tends to $+\infty$ and therefore, $|\theta|$ is bounded on $[0, +\infty)$ by a constant $\Theta$.
By definition $p=ta+\theta \sqrt{1-a^2}$, thus 
\[  \frac{t+ap}{\sqrt{1+a^2}}=t\sqrt{1+a^2}+\theta a \ge t\sqrt{1+a^2}-\Theta |a|\ge (t-\Theta)\sqrt{1+a^2}, \]
tends to $+\infty$ when $t\to +\infty$ (recall that $\Theta$ is a constant).
 Thus 
\begin{align*}
\lim_{t \to \infty} M\left( t,\Phi^{-1}(f(t)),\int_{-\infty}^{t} f d\gamma\right)= \lim_{t \to \infty}\varphi \left(\frac{p-at}{\sqrt{1+a^{2}}}\right)\Phi\left(\frac{t+ap}{\sqrt{1+a^{2}}}\right)=\varphi\left(\Phi^{-1}\left(\int_{-\infty}^{\infty}fd\gamma \right) \right).
\end{align*}
\end{proof}
\subsection{The proof of Bobkov's inequality}\label{karoche}
Let $B(t,x,y):=M(t,\Phi^{-1}(x),y)$ for $t \in \mathbb{R}$, $x \in (0,1)$ and $0<y<\Phi(t)$. Lemma~\ref{HJB} implies that 
\begin{align}\label{kac1}
I(x) \sqrt{\varphi^{2}(t)-B_{x}^{2}} = B_{t} + x \varphi(t) B_{y}.
\end{align}
One can easily check by studying the derivative in $v$ that 
\begin{align}\label{kac2}
\min_{v \in \mathbb{R}} \left\{ \varphi(t) \sqrt{I^{2}(x)+v^{2}} - v B_{x}  \right\} = I(x) \sqrt{\varphi^{2}(t)-B_{x}^{2}},
\end{align}
and that the minimum is  attained only when $v=\frac{I(x) B_{x}}{\sqrt{\varphi^{2}(t)-B_{x}^{2}}}$.
Therefore (\ref{kac1}) and (\ref{kac2}) imply that for any $v \in \mathbb{R}$ we have 
\begin{align}\label{kac3}
\varphi(t) \sqrt{I^{2}(x)+v^{2}} \geq B_{t}(t,x,y)+B_{x}(t,x,y) v + B_{y}(t,x,y) x \varphi(t),
\end{align}
where the inequality is strict when $v \neq\frac{I(x) B_{x}}{\sqrt{\varphi^{2}(t)-B_{x}^{2}}}$.

Now  take any $f \in C^{1}(\mathbb{R})$ with values in $(0,1)$ such that $\int_{\mathbb{R}}\sqrt{I(f)^{2}+(f')^{2}}d\gamma<\infty$ (otherwise there is nothing to prove). 
 Applying (\ref{kac3}) for  $x=f(t)$, $v=f'(t)$ and $y=\int_{-\infty}^t f\varphi$, we get: 
 \begin{align*}
 \Psi(t):=\sqrt{I^{2}(f(t))+(f'(t))^{2}} \varphi(t) - \frac{d}{dt}B\left( t,f(t), \int_{-\infty}^{t} f \varphi \right) \geq 0 \quad \text{for all} \quad t \in \mathbb{R}.
 \end{align*}
Therefore 
\begin{align*}
&\int_{-T}^{T}  \sqrt{I^{2}(f(t))+(f'(t))^{2}} \varphi(t)dt -M\left(T,\Phi^{-1}(f(T)),\int_{-\infty}^{T} f d\gamma\right)+
M\left(-T,\Phi^{-1}(f(-T)),\int_{-\infty}^{-T} f d\gamma\right)\\
&=\int_{-T}^{T} \left[ \sqrt{I^{2}(f(t))+(f'(t))^{2}} \varphi(t) - \frac{d}{dt}B\left( t,f(t), \int_{-\infty}^{t} f \varphi \right)  \right]dt =\int_{-T}^{T}\Psi(t) dt \geq 0. 
\end{align*} 

Finally sending $T \to \infty$ and using Lemma~\ref{endpoints} we obtain 
\begin{align}\label{kac4}
\int_{\mathbb{R}} \sqrt{I^{2}(f(t))+(f'(t))^{2}} \varphi(t) dt  -  I\left(\int_{\mathbb{R}} f \varphi \right)=\lim_{T \to \infty} \int_{-T}^{T}\Psi(t)dt \geq 0
\end{align}
Using standard approximation arguments we can extend (\ref{kac4}) to any  $C^{1}(\mathbb{R})$ smooth $f$ with values in $[0,1]$. This proves Bobkov's inequality (\ref{bob}) in  dimension $n=1$. To obtain (\ref{bob}) in an arbitrary dimension we use the standard tenzorization argument~\cite{BAMU}. Let us illustrate the argument for $n=2$. Take any $C^{1}(\mathbb{R}^{2})$ smooth $g(x, y)$ with values in $[0,1]$. We have 
\begin{align}
I\left(\int_{\mathbb{R}} \int_{\mathbb{R}}  g(x,y) d\gamma(x) d\gamma(y)\right) &\stackrel{(\ref{kac4})}{\leq} \int_{\mathbb{R}} \sqrt{I^{2}\left(\int_{\mathbb{R}}g(x,y)d\gamma(y) \right)+\left( \int_{\mathbb{R}}g_{x}(x,y)d\gamma(y)\right)^{2}} d\gamma(x) \nonumber\\
&\stackrel{(\ref{kac4})}{\leq} \int_{\mathbb{R}} \sqrt{ \left(\int_{\mathbb{R}} \sqrt{I^{2}(g)+g_{y}^{2}} d\gamma(y) \right)^{2}+\left( \int_{\mathbb{R}}g_{x}(x,y)d\gamma(y)\right)^{2}} d\gamma(x) \nonumber\\
\stackrel{\text{minkowski}}{\leq}\int_{\mathbb{R}}&\int_{\mathbb{R}} \sqrt{I^{2}(g) + g_{x}^{2}+g_{y}^{2}} d\gamma(x) d\gamma(y) = \int_{\mathbb{R}^{2}} \sqrt{I^{2}(g)+|\nabla g|^{2}} d\gamma^2. \label{franck2}
\end{align}
This finishes the proof of Bobkov's inequality. 
\subsection{Optimizers} Assume that a $C^1$ function $f:\R\to (0,1)$ is such that Bobkov's inequality \eqref{bob} is an equality. Then the left hand side of  \eqref{kac4} is zero. Since 
$\Psi$ is a non-negative continuous function,  it follows that $\Psi(t)=0$ for all $t \in \mathbb{R}$. This means that \eqref{kac3} was an equality when we applied it to prove that $\Psi\ge0$, therefore  $v=\frac{I(x) B_{x}}{\sqrt{\varphi^{2}(t)-B_{x}^{2}}}$ where $x=f(t)$, $v=f'(t)$, and $B_x$ stands for $B_x\big(t,f(t),\int_{-\infty}^t f\varphi\big)$.  Hence for all $t\in \mathbb R$, 
\begin{align*}
\frac{f'(t)}{I(f(t))} = \frac{B_{x}}{\sqrt{\varphi^{2}(t)-B_{x}^{2}}}.
\end{align*}
Let us rewrite this equation, by setting
 $h(t):=\Phi^{-1}(f(t))$ and using as before $M(t,p,y):=B(t,\Phi(p),y)$. Since
 $h'(t)=\frac{f'(t)}{I(f(t))}$ and $M_p(t,p,y)=\varphi(p) B_x(t,\Phi(p),y)$ we get
 after  simplification
\begin{align*}
h'(t) =\frac{M_{p}\Big(t,h(t),\int_{-\infty}^t \Phi(h)\varphi\Big)}{\sqrt{\varphi(t)^{2}\varphi(h(t))^{2}-M_{p}^{2}\Big(t,h(t),\int_{-\infty}^t \Phi(h)\varphi\Big)}}\stackrel{(\ref{Mp})}{=}\,a\Big(t,h(t),\int_{-\infty}^t \Phi(h)\varphi\Big).
\end{align*}
Since $a$ is  $C^1$, and so is $h$ by hypothesis, this equation shows that $h$ is $C^2$.
Using (\ref{franck})  we obtain 
\begin{align}\label{ura2}
\int_{-\infty}^{t} \Phi((s-t)h'(t)+h(t))\varphi(s) ds = \int_{-\infty}^{t}\Phi(h(s)) \varphi(s) ds.
\end{align}
After differentiation of  (\ref{ura2}) in $t$ and some simplifications we obtain 
\begin{align*}
h''(t) \int_{-\infty}^{t}\varphi((s-t)h'(t)+h(t))\varphi(s)(s-t) ds=0.
\end{align*}
The latter equality can hold if and only if $h''=0$, and thereby $f(t) = \Phi(ut+v)$ for some constants $u,v \in \mathbb{R}$. 

One can extend this result to higher dimensions by  showing that all $C^1$ functions $f:\mathbb R^n\to (0,1)$ which reach equality in Bobkov's inequality are of the form $f=\Phi\circ \ell$ for some linear form $\ell$. Indeed, for this we need to carefully examine the equality cases in the tensorization argument. Let us again illustrate the argument for $n=2$. Take any $g \in C^{1}(\mathbb{R}^{2})$ which takes values in $(0,1)$, and which achieves the equality in Bobkov's inequality.  Equality on the second step in the chain of inequalities (\ref{franck2}) implies that $g(x,y)=\Phi(yu(x)+v(x))$ for some functions $u(x), v(x)$. Since $g \in C^{1}$ and $\Phi$ is a smooth diffeomorphism we see that $u, v \in C^{1}(\mathbb{R})$. On the other hand equality in the part of Minkowski inequality (\ref{franck2}) implies that  
\begin{align*}
 \sqrt{I^{2}(g) + g_{y}^{2}} = k(x) g_{x}(x,y) 
\end{align*}
for a nonvanishing function $k(x)$. Simplifying the latter equality we obtain 
$$
\sqrt{1+u(x)^{2}} = k(x) (yu'(x)+v'(x)) \quad \text{for all} \quad  x,y \in \mathbb{R}.
$$
It follows that $u(x)=C_{1}$ is a constant, i.e, $g(x,y) = \Phi(yC_{1}+v(x))$. Repeating the same reasonings in a different order for variables $x,y$ one obtains that  $g(x,y)=\Phi(xC_{2}+\tilde{v}(y))$, and thereby $yC_{1}+v(x)=xC_{2}+\tilde{v}(y)$ for all $x,y \in \mathbb{R}$. Then it easily follows that $g(x,y)=\Phi(xC_{2}+yC_{2}+C_{3})$ for some constants $C_{1}, C_{2}$ and $C_{3}$.

Clearly,  these functions,  $f  = \Phi\circ \ell$ for some linear $\ell$, 
do give equality cases (the subgraph of $\Phi^{-1}\circ f=\ell$ is a half-space, which gives equality in the Gaussian isoperimetric inequality).
 However our approach  at the current stage is not well developed. Carlen and Kierce \cite{CK} have studied equality cases in the natural larger class of functions with bounded variations, where additional equality cases are given by indicator functions of half-spaces.

\section{Concluding remarks}
We briefly sketch to the reader how the argument of optimal control theory works in general. Suppose we would like to 
 maximize   the quantity  
\begin{align}\label{optim0}
\int_{\mathbb{R}}F\big(t,f(t), f'(t)\big) \, dt
\end{align} 
in terms of $\int_{\mathbb{R}} H(t,f(t))dt$ where $F$ and $H$ are some given functions,   $f$ is a test function from a {\em sufficiently nice class}  so that all the expressions involved are well defined. Clearly this means that we would like to solve the following optimization problem 
\begin{align*}
R(y):= \sup_{f} \left\{ \int_{\mathbb{R}}F\big(t,f(t), f'(t)\big)\,dt\; :\; \int_{\mathbb{R}} H\big(t,f(t)\big)\, dt = y \right\}.
\end{align*}
Unfortunately the function $R(y)$ may not obey good properties, for example it is unclear how to find the corresponding ODE that $R(y)$ would satisfy.  Therefore, following the optimal control theory approach,   we should  introduce some extra variables, namely, we should first consider a more general optimization problem
\begin{align}\label{OOP}
B(t,x,y) :=\sup_{f} \left\{ \int_{-\infty}^{t} F\big(s,f(s), f'(s)\big)\, ds\; : \; f(t)=x, \; \int_{-\infty}^{t} H\big(s,f(s)\big) \,ds=y \right\}.
\end{align}
Then the limit value $\sup_{x} \lim_{t \to \infty} B(t,x,y)$ would be a good candidate for $R(y)$. On the other hand  using the standard Bellman principle (see for example \cite{young1969}) one can show that 
\begin{align}\label{BP}
F(t,x,v)\leq B_{t}(t,x,y)+B_{x}(t,x,y)v+B_{y}(t,x,y)H(t,x)
\end{align}
for all $v \in\mathbb{R}$. Indeed, take any $(t,x,y)$ and assume $f^{*}(s)$ optimizes (assume it exists) the right hand side of (\ref{OOP}) on the interval  $(-\infty,t]$ with fixed $f^{*}(t)=x$ and $\int_{-\infty}^{t} H(s,f^{*}(s))ds=y$, then take a small $\varepsilon >0$, any $v\in \mathbb{R}$, and construct a new candidate on $(-\infty,t+\varepsilon]$, namely, 
\begin{align*}
\tilde{f}(s) = 
\begin{cases}
f^{*}(s), \quad  s \leq t;\\
f^{*}(t)+v(s-t), \quad   s \in [t,t+\varepsilon]. 
\end{cases}
\end{align*}
Then 
\begin{align*}
&B\left(t+\varepsilon,x+v \varepsilon, y+\int_{t}^{t+\varepsilon} H\big(s,x+v(s-t)\big)ds\right)=B\left(t+\varepsilon, \tilde{f}(t+\varepsilon), \int_{-\infty}^{t+\varepsilon} H\big(s,\tilde{f}(s)\big)ds\right) \geq\\
&\int_{-\infty}^{t+\varepsilon} F\big(s,\tilde{f}(s),\tilde{f}'(s)\big)ds = 
B(t,x,y)+\int_{t}^{t+\varepsilon}F\big(s,x+v(s-t),v\big)ds.
\end{align*}   
Subtracting $B(t,x,y)$ from both sides of the latter inequality,  dividing by $\varepsilon$ and sending $\varepsilon$ to zero we arrive at (\ref{BP}). Here we are omitting several details and assumptions, for example, $B$ does not have to be differentiable. 

 On the other hand if one finds any function $\tilde{B}(t,x,y)$ such that (\ref{BP}) holds with $\tilde{B}$ instead of $B$, and $\tilde{B}$ has the additional property that 
 $$
 \lim_{t \to -\infty} \tilde{B}\left( t,f(t),\int_{-\infty}^{t} H\big(s,f(s)\big)ds\right)=0,
$$ 
  then one automatically obtains the bound  $\tilde{B}\geq B$.   Indeed, take $f(t)$, and notice that (\ref{BP}) for $\tilde{B}$ implies  
\begin{align*}
F\big(s,f(s),f'(s)\big)\leq \frac{d}{ds} \tilde{B}\left(s,f(s), \int_{-\infty}^{s} H\big(u,f(u)\big)du\right).
\end{align*}  
Now integrating  in $s$ on the ray  $(-\infty,t]$ we obtain that 
\begin{align}\label{least1}
B \leq \tilde{B}.
\end{align}
So we see that the problem of solving (\ref{OOP}) boils down to finding solutions of (\ref{BP}). We can optimize (\ref{BP}) in $v$, i.e., 
\begin{align}\label{BEN}
\sup_{v \in \mathbb{R}} \left\{ F(t,x,v) - B_{x}(t,x,y)v\right\} \leq B_{t}(t,x,y) +B_{y}(t,x,y)H(t,x)
\end{align}
Since $B$ should be the least (\ref{least1}) such possible solution it is quite natural to expect that in fact we should have equality in (\ref{BEN}) instead of inequality. Thus we arrive to the first order fully nonlinear PDE, the so called Hamilton--Jacobi--Bellman PDE, which can be solved by the methods of characteristics.   

To summarize we should mention that the function $B$ that we found in Section~\ref{karoche}  is the solution of the following optimization problem\footnote{Here we have infimum  instead of supremum but the reader can notice that all the reasonings described above will repeat absolutely in the same way except all inequalities will be reversed and $\sup$ in (\ref{BEN}) will be replaced by $\inf$.}
\begin{align}\label{BBE}
B(t,x,y) = \inf_{f \in C^{1}} \left\{ \int_{-\infty}^{t} \sqrt{I^{2}(f(s))+(f'(s))^{2}} \varphi(s)ds\; :\; f(t)=x, \; \int_{-\infty}^{t} f(s)\varphi(s)ds=y \right\}.
\end{align} 
Next we made a shortcut in solving (\ref{BEN}), for example, one can guess from the Euler--Lagrange equation that the optimizers in (\ref{BBE}) should be $f(s)=\Phi(as+b)$ for two arbitrary constants $a,b \in \mathbb{R}$ (on can also argue that global extremizers $f$ in Bobkov's inequality should be such that the subgraph of $\Phi^{-1}\circ f$ is a half-space, for which the Gaussian isoperimetric inequality is tight). We can use this information in order to immediately recover the function $B(t,x,y)$. Indeed, first we find $a=a(t,x,y)$ and $b=b(t,x,y)$ such that $\Phi(at+b)=x$, and $\int_{-\infty}^{t}\Phi(as+b)\varphi(s)ds=y$. Plugging $f(s):=\Phi(a(t,x,y)s+b(t,x,y))$ into the functional of the right hand side in (\ref{BBE}) recovers the function $B(t,x,y)$.

\end{document}